\documentclass{article}

\usepackage{titling}

\usepackage{amsfonts,amsmath,amssymb,amsthm}
\usepackage{mathscinet}
\usepackage{mathrsfs}

\usepackage{combelow}

\usepackage{enumitem}
\setitemize{itemsep=0em}

\usepackage{multirow}

\usepackage[font=small,labelfont=bf]{caption}

\usepackage{todonotes}

\newcommand{\email}[1]{\href{mailto:#1}{#1}}

\usepackage[pdfusetitle]{hyperref}
\usepackage{color}
\definecolor{darkred}{rgb}{.8,0,0}
\definecolor{tocolor}{rgb}{.1,.1,.1}
\definecolor{urlcolor}{rgb}{.2,.2,.6}
\definecolor{linkcolor}{rgb}{.1,.1,.5}
\definecolor{citecolor}{rgb}{.4,.2,.1}
\definecolor{gray}{rgb}{.8,.8,.8}

\hypersetup{
	colorlinks=true,
	urlcolor=urlcolor,
	linkcolor=linkcolor,
	citecolor=citecolor
	}

\usepackage{aliascnt}

\newcommand{\thdef}[2]{
	\newaliascnt{#1}{theorem}
	\newtheorem{#1}[#1]{#2}
	\aliascntresetthe{#1}
	\newtheorem*{#1*}{#2}
	\expandafter\newcommand\expandafter{\csname #1autorefname\endcsname}{#2}
}

\newtheorem{theorem}{Theorem}[section]
\newtheorem*{theorem*}{Theorem}
\thdef{conjecture}{Conjecture}
\thdef{lemma}{Lemma}
\thdef{corollary}{Corollary}
\thdef{proposition}{Proposition}
\theoremstyle{definition}
\thdef{definition}{Definition}

\theoremstyle{remark}
\thdef{remarkx}{Remark}
\thdef{examplex}{Example}
\thdef{setup}{Setup}
\thdef{claim}{Claim}

\newenvironment{example}
  {\pushQED{\qed}\examplex}
  {\popQED\endexamplex}

\newenvironment{remark}
  {\pushQED{\qed}\remarkx}
  {\popQED\endremarkx}

\usepackage{tikz,tikz-cd}

\newcommand{\rbrac}[1]{\left(#1\right)}

\newcommand{\CC}{\mathbb{C}}
\newcommand{\PP}{\mathbb{P}}

\newcommand{\QQ}{\mathbb{Q}}

\newcommand{\T}{\mathsf{T}}
\newcommand{\F}{\mathsf{F}}

\newcommand{\A}{\mathsf{A}}
\newcommand{\Z}{\mathsf{Z}}
\newcommand{\Y}{\mathsf{Y}}
\newcommand{\X}{\mathsf{X}}
\newcommand{\Si}{\mathsf{S}}
\newcommand{\tX}{\widetilde{\mathsf{X}}}

\newcommand{\U}{\mathsf{U}}

\newcommand{\bL}{\mathsf{L}}

\newcommand{\kahcl}{\gamma}

\newcommand{\M}{\mathsf{M}}
\newcommand{\G}{\mathsf{G}}
\newcommand{\N}{\mathsf{N}}

\newcommand{\Aut}[1]{\mathsf{Aut}\rbrac{#1}}

\DeclareMathOperator{\rank}{rank}
\DeclareMathOperator{\img}{img}

\newcommand{\cT}[1]{\mathcal{T}_{#1}} 
\newcommand{\cN}[1]{\mathcal{N}_{#1}} 

\newcommand{\forms}[2][\bullet]{\Omega^{#1}_{#2}} 
\newcommand{\cF}{\mathcal{F}}
\newcommand{\cG}{\mathcal{G}}
\newcommand{\cD}{\mathcal{D}}
\newcommand{\cE}{\mathcal{E}}
\newcommand{\cL}{\mathcal{L}}

\newcommand{\cO}[1]{\mathcal{O}_{#1}} 

\newcommand{\coH}[2][\bullet]{\mathsf{H}^{#1}(#2)}
\newcommand{\coh}[2][\bullet]{\mathsf{h}^{#1}(#2)}

\newcommand{\End}[1]{\mathsf{End}(#1)}

\newcommand{\hook}[1]{\iota_{#1}}

\newcommand{\PoissonVanishingConjecture}{Let $(\X,\pi)$ be a compact K\"{a}hler Poisson manifold, and let $d$ be the minimal dimension of a symplectic leaf of $(\X,\pi)$.  Then the Hodge numbers $h^{2j,0}(\X)$ are nonzero for all $j \leq \tfrac{d}{2}$.  In particular, if $\coh[2,0]{\X} = 0$, then every holomorphic Poisson structure on $\X$ has a zero.}

\newcommand{\PoissonSplittingConjecture}{Let $(\X,\pi)$ be a compact K\"ahler Poisson manifold and suppose that the minimal dimension of a symplectic leaf of $\pi$ is equal to $d$.  Then there exists a holomorphic symplectic manifold $\Y$ of dimension $d$, a holomorphic Poisson manifold $\Z$, and a holomorphic Poisson covering map $\Y\times \Z \to \X$.}

\begin{document}

\title{Numerically flat foliations and holomorphic Poisson geometry}

\author{St\'ephane Druel\thanks{CNRS/Universit\'e Claude Bernard Lyon 1, \email{stephane.druel@math.cnrs.fr}} \and Jorge Vit\'orio Pereira\thanks{IMPA, \email{jvp@impa.br}} \and Brent Pym\thanks{McGill University, \email{brent.pym@mcgill.ca}} \and Fr\'ed\'eric
Touzet\thanks{Universit\'e de Rennes , \email{frederic.touzet@univ-rennes.fr}}}

\maketitle

{\centering\textit{To Jean-Pierre Demailly, in memoriam.} \par}

\begin{abstract}
    We investigate the structure of smooth holomorphic foliations with numerically flat tangent bundles on compact K\"ahler manifolds. Extending earlier results on non-uniruled projective manifolds by the second and fourth authors, we show that such foliations induce a decomposition of the tangent bundle of the ambient manifold, have leaves uniformized by Euclidean spaces, and have torsion canonical bundle. Additionally, we prove that  smooth two-dimensional foliations with numerically trivial canonical bundle on projective manifolds are either isotrivial fibrations or have numerically flat tangent bundles.  This in turn implies a global Weinstein splitting theorem for rank-two Poisson structures on projective manifolds.  We also derive new Hodge-theoretic conditions for the existence of zeros of Poisson structures on compact K\"ahler manifolds.
\end{abstract}

\renewcommand{\abstractname}{R\'esum\'e}
\begin{abstract} Dans cet article, nous nous int\'eressons à la structure des feuilletages holomorphes r\'eguliers  dont le fibr\'e tangent est num\'eriquement plat, la vari\'et\'e ambiante \'etant compacte k\"ahlérienne. Nous \'etendons dans ce cadre des r\'esultats pr\'ec\'edemment obtenus par les deuxième et quatrième auteurs. Nous montrons notamment que l'existence d'un tel feuilletage induit une d\'ecomposition du fibr\'e tangent de la vari\'et\'e, que les feuilles sont uniformis\'ees par un espace euclidien et que le fibr\'e canonique dudit feuilletage est de torsion. En outre, nous \'etablissons, lorsque la vari\'et\'e ambiante est suppos\'ee projective,  qu'un feuilletage r\'egulier de dimension deux dont le fibr\'e canonique est num\'eriquement trivial  est ou bien  une fibration isotriviale, ou bien possède un fibr\'e tangent num\'eriquement plat. Ce dernier r\'esultat fournit un analogue global du th\'eor\`eme de d\'ecomposition de Weinstein pour les structures de Poisson de rang deux sur les vari\'et\'es projectives lisses. Nous obtenons également de nouvelles conditions sur les nombres de Hodge pour qu'une structure de Poisson sur une variété compacte k\"ahlérienne s'annule en un point.
\end{abstract}

\section{Introduction}

In this paper, we establish several results and conjectures concerning the structure of holomorphic foliations $\cF$ on compact K\"ahler manifolds $\X$, under suitable cohomological vanishing conditions on the curvature of the tangent sheaf $\cT{\cF}$. Our main results give conditions for such foliations to be induced by a splitting of the universal cover of $\X$ into a product of manifolds.  As an application, we obtain some interesting consequences for the structure of holomorphic Poisson brackets on compact K\"ahler manifolds.

\subsection{Numerically flat foliations}
Our first main result, established in \autoref{S:proof main}, describes the structure of regular foliations whose tangent bundle is numerically flat in the sense of Demailly--Peternell--Schneider~\cite{DPS94}, a notion we recall in \autoref{sec:num-flat}.

Such foliations were previously analyzed by the second and fourth authors in \cite{Pereiratouzet2013}, under the additional assumptions that the ambient manifold $\X$ is projective but not uniruled.  Remarkably, with these additional assumptions,  the vanishing of the first Chern class of the foliation  implies both the smoothness of the foliation and the polystability of its tangent bundle. This is established in \cite[Lemma 2.1]{Pereiratouzet2013} (see also \cite[Section 5]{Loray2018}), building on Demailly's integrability theorem for differential forms with coefficients in the dual of a pseudoeffective line bundle \cite{Demailly02} and the characterization of non-uniruled projective manifolds by Boucksom--Demailly--P\u{a}un--Peternell \cite{BDPP}. Leveraging these two properties, it is further shown (\textit{loc.\ cit.}) that the canonical bundle of the foliation is a torsion line bundle, its leaves are uniformized by Euclidean spaces, and their analytic closures are quotients of abelian varieties. Similarly, and almost concurrently with \cite{Pereiratouzet2013}, the work \cite{Amoros} provides a precise description of smooth foliations with trivial tangent bundle on arbitrary compact K\"ahler manifolds, extending previous work by Lieberman \cite{Lieberman78}.

In this paper, we extend the results from \cite{Pereiratouzet2013} to smooth foliations with numerically flat tangent bundle on arbitrary compact K\"ahler manifolds.

\begin{theorem}\label{thm:num_flat_new_main}
    Let $\X$ be a compact K\"{a}hler manifold, and let $\cF$ be a regular foliation of dimension $p\ge 0$ on $\X$ with numerically flat tangent bundle $\cT{\cF}$. Then the following hold.
    \begin{enumerate}
        \item The tangent bundle of $\cF$ is hermitian flat, and the line bundle $\det \cT{\cF}$ is torsion.
        \item There exists a foliation $\cG$ on $\X$  such that $\cT{\X} = \cT{\cF} \oplus \cT{\cG}$.
        \item The universal cover $\tX$ of $\X$ decomposes as a product $\tX \cong \CC^p \times \Y$ where $\Y$ is a complex manifold, and the decomposition $\cT{\X}=\cT{\cF}\oplus \cT{\cG}$ lifts to the canonical decomposition $\cT{\widetilde{\X}} \cong  \cT{\CC^p} \boxplus\cT{\Y}$.
        \item The analytic closure $\overline{\bL}\subseteq \X$ of any leaf $\bL$ of $\cF$ is isomorphic to a finite \'etale quotient of an equivariant compactification of an abelian Lie group $\G_{\bL}$, in such a way that the foliation $\cF|_{\overline{\bL}}$ is induced by a (not necessarily closed) subgroup of $\G_{\bL}$.
    \end{enumerate}
\end{theorem}

In the case of uniruled manifolds, smoothness does not follow automatically from the triviality of the first Chern class.  For example, one-dimensional foliations with trivial first Chern class on simply connected manifolds are necessarily singular, presenting  a significant obstacle to adapting the arguments of \cite{Pereiratouzet2013} to our current setting.  We get around this using an alternative approach to constructing splittings of the tangent bundle that exploits the behaviour of the K\"ahler class along the leaves of the foliation; see \autoref{sec:intro-cohK} below.

\subsection{Foliations with numerically trivial canonical bunlde}

Our second main result, established in \autoref{sec:dim2-CY}, describes the structure of smooth foliations of dimension two with numerically trivial canonical bundle on projective manifolds. For projective threefolds, these were described  using Mori Theory in \cite{DruelPoisson}. Later,  a similar description was obtained for smooth codimension one foliations on compact K\"ahler manifolds using entirely different methods in \cite{CalabiYauTouzet}. In this work, we establish the following result.

\begin{theorem}\label{thm:second_main}
Let $\X$ be a complex projective manifold, and let $\cF$ be a regular foliation of dimension two with $c_1(\cT{\cF})=0$. Then the canonical bundle of $\cF$ is a torsion line bundle. Moreover, either  $\cT{\cF}$ is hermitian flat, or $\X$ has a finite \'etale cover that decomposes as a product $\bL\times\Y$ where $\bL$ is a surface with zero first Chern class, $\Y$ is a complex projective manifold, and the foliation is induced by the projection to $\Y$.
\end{theorem}

Unlike in \autoref{thm:num_flat_new_main}, here we are forced to restrict to projective manifolds due to the lack of compact K\"ahler analogues
of the algebraicity/compactness criteria for leaves used in the proof of \autoref{thm:second_main}.

\subsection{Cohomologically K\"ahler foliations and splittings}
\label{sec:intro-cohK}
A recurring theme in our arguments, and in previous works such as \cite{Druel2020,globalweinstein,Loray2018}, is the behaviour of the restriction of the K\"ahler class to the leaves of the foliation.  We have found it useful to isolate the key property in the following definition, which we introduce and develop in \autoref{S:splitting}.
\begin{definition}
Let $\cF$ be a foliation on a compact K\"ahler manifold $\X$, and let $p=\dim\cF$.  We say that $\cF$ is \emph{cohomologically K\"ahler} if for every K\"ahler class $\kahcl \in \coH[1]{\X,\forms[1]{\X}}$, the image of $\kahcl^p$ in $\coH[p]{\X,\omega_\cF}$ is nonzero, where $\omega_\cF = (\wedge^p \cT{\cF})^*$ is the canonical line bundle of $\cF$.
\end{definition}
This condition is readily checked in many cases, e.g.~it is stable along \'etale covers and embeddings, and holds automatically for codimension one foliations; the most subtle results we obtain in this direction rely on Demailly's integrability theorem~\cite{Demailly02}. In fact, we conjecture that this condition holds for all regular foliations on compact K\"ahler manifolds (\autoref{conj:regular-foliation-cK}). Meanwhile, this condition is closely linked to the existence of a subbundle of $\cT{\X}$ that is complementary to $\cT{\cF}$.  In particular, the existence of such a complement easily implies that the foliation is cohomologically K\"ahler (\autoref{lem:direct-sum-cK}), and the converse holds if $c_1(\cT{\cF})=0$ (\autoref{prop:main_lemma}).

\subsection{Applications to Poisson geometry}

In \autoref{sec:poisson}, we explain some consequences of our results and conjectures for holomorphic Poisson structures. On the one hand, they suggest the following global version of  Weinstein's splitting theorem~\cite{Weinstein1983}, generalizing our results in \cite{globalweinstein} (which treated the case of possibly singular Poisson structures with a simply-connected compact leaf):

\begin{conjecture}\label{conj:weinstein}
\PoissonSplittingConjecture
\end{conjecture}
Note that in the conjecture, the covering map is allowed to be infinite, i.e.~$\Y$ or $\Z$ may be noncompact.

As we explain in and around \autoref{prop:weinstein}, this conjecture was previously known to hold when the corank of $\pi$ is at most one~\cite{DruelPoisson,CalabiYauTouzet}, and our \autoref{thm:second_main} above implies that it also holds when the minimal dimension of a leaf is equal to two, provided $\X$ is projective.  In summary, these results give the following.

\begin{theorem}
\autoref{conj:weinstein} holds for all projective manifolds of dimension  $\dim \X \leq 5$.
\end{theorem}

On the other hand, we conjecture a tight relationship between the Hodge numbers of $\X$ and the existence of zeros of the Poisson structure, similar to Bondal's conjecture~\cite{Bondal1993} on the dimensions of degeneracy loci of Poisson structures on Fano manifolds.

\begin{conjecture}\label{conj:poisson-vanishing}
\PoissonVanishingConjecture
\end{conjecture}

We show that this conjecture would follow from our conjecture that all regular foliations are cohomologically K\"ahler.  Exploiting our results in this paper and previous work on Bondal's conjecture, \autoref{prop:poisson-vanishing} and \autoref{prop:Fano-vanishing} together give the following.
\begin{theorem}
\autoref{conj:poisson-vanishing} holds  if $\X$ is projective and $\dim \X \leq 6$; if $\X$ is Fano and $\dim \X = 7$; or if $\X$ is Fano, $\dim \X = 8$ and $b_2(\X)=1$.
\end{theorem}

\paragraph{Acknowledgements:} J.V.P. and S.D. acknowledge support from  the CAPES-COFECUB project Ma1017/24, funded by the French Ministry for Europe and Foreign Affairs, the French Ministry for Higher Education and CAPES.  J.V.P.~was also supported by CNPq (PQ Scholarship 304690/2023-6 and Projeto Universal 408687/2023-1 “Geometria das Equações Diferenciais Alg\'ebricas”), and FAPERJ (Grant number E26/200.550/2023). B.P.~was supported by a faculty startup grant at McGill University, a New university researchers startup grant from the Fonds de recherche du Qu\'ebec -- Nature et technologies (FRQNT), and by the Natural Sciences and Engineering Research Council of Canada (NSERC), through Discovery Grant RGPIN-2020-05191.  F.T. would like to thank Henri Lebesgue Center for constant support.

\section{Notation, conventions, and useful facts}\label{section:notation}

\subsection{Global conventions} Throughout the paper, all algebraic varieties are assumed to be defined over the field of complex numbers. We will freely switch between the algebraic and analytic context.

\subsection{Stability} The word \textit{stable} will always mean \textit{slope-stable with respect to a
given ample divisor}, and similarly for \textit{semistable} and \textit{polystable}.
We refer to \cite[Definition 1.2.12]{HuyLehn} for their precise definitions.

\subsection{Numerically flat vector bundles}\label{sec:num-flat}
 One key notion is that of a numerically flat vector bundle. We recall the definition following \cite[Definition 1.17]{DPS94}.

\begin{definition}
A vector bundle $\cE$ of rank $r\ge 1$ on a compact K\"{a}hler manifold is called \textit{numerically flat} if $\cE$ and $\cE^*$ are nef vector bundles. Equivalently, $\cE$ is numerically flat if and only if $\cE$ and $\det \cE ^*$ are nef vector bundles.
\end{definition}

\begin{remarkx}\label{rem:num_flat_zero_first_Chern_class}
Let $\X$ be a compact K\"{a}hler manifold, and let $\cE$ be a vector bundle of rank $r \ge 1$ on $\X$ with $c_1(\cE) = 0 \in \coH[2]{\X,\mathbb{C}}$. Then $\cE$ is numerically flat if and only if $\cE$ is nef.
\end{remarkx}

\begin{theorem}[{\cite[Theorem 1.18]{DPS94}}]\label{thm:numerically_flat}
Let $\X$ be a compact K\"{a}hler manifold, and let $\cE$ be a vector bundle of rank $r \ge 1$ on $\X$. Then $\cE$ is numerically flat if and only if $\cE$ admits a filtration
$$\{0\}=\cE_0 \subsetneq \cE_1 \subsetneq \cdots \subsetneq \cE_m=\cE$$
by vector subbundles such that the quotients $\cE_k/\cE_{k-1}$ are hermitian flat.
\end{theorem}

\begin{corollary}[{\cite[Corollary 1.19]{DPS94}}]
Let $\cE$ be a vector bundle on a compact K\"ahler manifold $\X$. If $\cE$ is numerically flat, then $c_i(\cE) = 0 \in \coH[2i]{\X,\mathbb{C}}$ for any integer $i \ge 1$.
\end{corollary}

\begin{remarkx}
A numerically flat vector bundle is automatically flat \cite[Section 3]{Simpson92} but is not hermitian flat in general \cite[Example 1.7]{DPS94}.
\end{remarkx}
\subsection{Foliations}

\begin{definition}
A \emph{foliation} $\cF$ on a complex manifold $\X$ is determined by its tangent sheaf $\cT{\cF}$, which is a  coherent subsheaf  of $\cT{\X}$ such that
\begin{enumerate}
\item $\cT{\cF}$ is closed under the Lie bracket, and
\item $\cT{\cF}$ is saturated in $\cT{\X}$, i.e.~the quotient $\cT{X}/\cT{\cF}$ is torsion-free.
\end{enumerate}
The \emph{dimension} $p$ of $\cF$ is the generic rank of $\cT{\cF}$.
The \emph{codimension} of $\cF$ is defined as $q=\dim \X-p$.  The \emph{normal sheaf} of $\cF$ is $\cN{\cF}=(\cT{\X}/\cT{\cF})^{**}$; by definition, it is a torsion-free sheaf of rank $q$.

\medskip

The cotangent sheaf $\Omega_\cF^1$ of $\cF$ is $\Omega_\cF^1=\cT{\cF}^*$. Its \textit{canonical line bundle} $\omega_\cF$ is $\omega_\cF = \det \Omega_\cF^1 = (\wedge^p\,\Omega_\cF^1)^{**}$.  Note that the canonical map $\wedge^p \cT{\cF} \to \wedge^p \cT{\X}$ determines a section
\[
v \in \coH[0]{\X,\wedge^p \cT{\X} \otimes \omega_\cF},
\]
from which we may recover $\cT{\cF}$ as the sheaf of vector fields $\xi \in \cT{\X}$ such that $v \wedge \xi=0$.  We say that $v$ is a \emph{$p$-vector defining $\cF$}.  Dually, we may define $\cF$ by a form
\[
\alpha \in \coH[0]{\X,\forms[q]{\X}\otimes \omega_\cF^*}
\]
whose kernel is $\cT{\cF}$.

\medskip

Let $\X^\circ \subseteq \X$ be the open set where $\cT{\cF}$ is a subbundle of $\cT{\X}$. The \emph{singular locus of $\cF$} is
\[
\mathrm{Sing}(\cF) := \X \setminus \X^\circ.
\]
The foliation is called \emph{regular}, or \emph{smooth}, if its singular locus is empty.

\medskip

A \emph{leaf} of $\cF$ is a maximal connected and immersed holomorphic submanifold $\bL \subseteq \X^\circ$ such that $\cT{\bL}=\cT{\cF}|_{\bL} \subseteq \cT{X^\circ}|_{\bL}$.
\end{definition}

\begin{lemma}\label{lem:invariant}
Let $\Y \subseteq \X$ be a submanifold that is not contained in the singular locus of $\cF$.  Then the following statements are equivalent:
\begin{enumerate}
\item For every $y\in \Y \cap \X^\circ$, the leaf through $y$ is contained in $\Y$.
\item The defining $p$-vector is tangent to $\Y$, i.e.
\[v|_\Y \in \coH[0]{\Y,\wedge^p\cT{\Y}\otimes\omega_\cF|_\Y} \subseteq
\coH[0]{\Y,\wedge^p\cT{\X}|_\Y\otimes\omega_\cF|_\Y}.\]
\item The contraction $\iota_v|_\Y : \forms[p]{\X}|_\Y \to \omega_\cF|_\Y$ factors through the natural surjection $\forms[p]{\X}|_\Y \to \forms[p]{\Y}$.
\end{enumerate}
\end{lemma}

\begin{definition}
Let $\Y \subseteq \X$ be a submanifold that is not contained in the singular locus of $\cF$.  We say that $\Y$ is  \emph{$\cF$-invariant} if  any (hence all) of the equivalent properties of \autoref{lem:invariant} hold.
\end{definition}

By relaxing the condition of involutiveness on the tangent sheaf of a foliation, we arrive at the broader notion of a distribution.

\begin{definition} A \emph{distribution} $\cD$ on a complex manifold $\X$ is specified by its tangent sheaf $\cT{\cD}$, which is a coherent subsheaf of $\cT{\X}$, saturated in $\cT{\X}$, meaning that the quotient $\cT{\X}/\cT{\cD}$ is torsion-free. \end{definition}

All the concepts introduced for foliations earlier in this section -- such as dimension, codimension, cotangent sheaf, singular set, invariant submanifold,...  -- naturally extend to distributions. While every foliation is a distribution, the converse is not true. The key difference is that Frobenius theorem guarantees the existence of a leaf through any point outside the singular locus of a foliation, whereas an arbitrary distribution might have no leaf at all.

\section{Cohomologically K\"ahler foliations and split tangent bundles}\label{S:splitting}

\subsection{Cohomologically K\"ahler foliations}\label{subsection:cKf}
Throughout this section, $\X$ is a compact K\"ahler manifold of dimension $n$ and $\cD$ (resp.\ $\cF$) is a holomorphic distribution (resp.\ foliation) on $\X$ of dimension $p$, which we allow to be singular unless otherwise stated. We denote by $q=n - p$ its codimension.

The key notion is the following cohomological condition on $\cD$ or $\cF$, which is modeled on the behaviour of K\"ahler classes under restriction to submanifolds.

\begin{definition}\label{def:positive-invariant-subspace}
We say that $\cD$ is \textit{cohomologically K\"ahler}
if, for any K\"ahler class $\kahcl \in \coH[1]{\X,\forms[1]{\X}}\cong\coH[1,1]{\X}$, the image of $\kahcl^p$ under the natural map
\[
\hook{v} : \coH[p]{\X,\forms[p]{\X}} \to \coH[p]{\X,\omega_\cD}
\]
is nonzero.

We say that $\cD$ is \textit{weakly cohomologically K\"ahler}
if there exists a K\"ahler class $\kahcl \in \coH[1]{\X,\forms[1]{\X}}\cong\coH[1,1]{\X}$ such that
$\hook{v} \kahcl^p$ is  nonzero.
\end{definition}

\begin{example}\label{ex:extreme-cohK}
At the extremes, the foliations with $\cT{\cF}=0$ (respectively, $\cT{\cF}=\cT{\X}$) whose leaves are points (resp.~all of $\X$) are cohomologically K\"ahler.
\end{example}

We now give some useful sufficient conditions for a distribution $\cD$ to be cohomologically K\"ahler, starting with the following easy observations.

\begin{lemma}\label{lem:invariant-subspace-kahler}
If $\Y \subseteq \X$ is a closed $\cD$-invariant submanifold, and $\cD|_\Y$ is cohomologically K\"ahler, then $\cD$ itself is cohomologically K\"ahler.
\end{lemma}

\begin{proof}
If $\Y \subseteq \X$ is a $\cD$-invariant subspace, then the restriction map $\forms[p]{\X} \to \omega_{\cD}|_\Y\cong \omega_{\cD|_\Y}$ factors through $\forms[p]{\Y}$, which immediately implies the result.
\end{proof}

\begin{lemma}\label{lem:topological-cover-kahler}
If $f\colon \Y \to \X$ is a finite \'etale cover, and $f^{-1}\cD$ is (weakly) cohomologically K\"ahler, then $\cD$ itself is (weakly) cohomologically K\"ahler.
\end{lemma}
\begin{proof}
Let $\cE:=f^{-1}\cD$. Notice that there is a commutative diagram
\begin{center}
\begin{tikzcd}[row sep=large]
\coH[p]{\Y,\forms[p]{\Y}} \ar[r, "\hook{v_\cE}"] & \coH[p]{\Y,\omega_\cE}\\
\coH[p]{\X,\forms[p]{\X}}\ar[u] \ar[r, "\hook{v_\cD}"] & \coH[p]{\X,\omega_\cD}\ar[u].
\end{tikzcd}
\end{center}
and every K\"ahler class on $\X$ pulls back to a K\"ahler class on $\Y$. This immediately implies the result.
\end{proof}

\begin{example}\label{ex:torus-quotient-kahler}
If $\X$ is a finite \'etale quotient of a torus $\T$, and $\cF|_\X$ is a linear foliation, i.e.~its pullback to $\T$ is induced by the action of a (not necessarily closed) subgroup of $\T$,
then $\cF$ is cohomologically K\"ahler.  Indeed, in this case, the pullback of $\cT{\cF}$ to $\T$ is trivial, and the pullback of any K\"ahler class $\kahcl$ can be represented by a Hermitian form that is nondegenerate on every trivial subbundle of $\cT{\T}$. The claim now follows from \autoref{lem:topological-cover-kahler}.
\end{example}

\begin{lemma}\label{lem:direct-sum-cK}
If there exists a distribution $\cE$ on $\X$ such that $\cT{\X}=\cT{\cD}\oplus\cT{\cE}$, then $\cD$ is cohomologically K\"ahler. 
\end{lemma}

\begin{proof}
Let $\kahcl = \kahcl_\cD + \kahcl_\cE \in \coH[1]{\X,\forms[1]{\X}} = \coH[1]{\X,\forms[1]{\cD}} \oplus \coH[1]{\X,\forms[1]{\cE}}$. Then $\kahcl_\cD^{p+1}=0 \in \coH[p+1]{\X,\forms[p+1]{\cD}}$ and $\kahcl_\cE^{q+1}=0 \in \coH[q+1]{\X,\forms[q+1]{\cE}}$. As a consequence, $\kahcl^n={n \choose p}\cdot   \kahcl_\cD^{p}\wedge \kahcl_\cE^{q}$. On the other hand,  $\kahcl_\cD^p = \hook{v} \kahcl^p$. In particular, if $\hook{v} \kahcl^p= 0 \in \coH[p]{\X,\omega_\cD}$, then $\kahcl^n=0$, which is impossible since $\kahcl$ is a K\"ahler class and $\X$ is compact. This proves that $\cD$ is cohomologically K\"ahler.
\end{proof}

\begin{lemma}
If $\cD$ is regular of dimension $p=1$, then $\cD$ is cohomologically K\"ahler.
\end{lemma}

\begin{proof}
Let $\kahcl \in \coH[1]{\X,\forms[1]{\X}}\cong\coH[1,1]{\X}$ be a K\"ahler class. Suppose by way of contradiction that
$\hook{v} \kahcl= 0 \in \coH[p]{\X,\omega_\cF}$, then $\kahcl$ lies in the image of the natural map
\[
\coH[1]{\X,\cN{\cF}^*} \to \coH[1]{\X,\forms[1]{\X}}.
\]
It follows that $\kahcl^n=0$, which is impossible since $\kahcl$ is a K\"ahler class and $\X$ is compact.
\end{proof}

Actually, we expect that regular foliations are always cohomologically K\"ahler:

\begin{conjecture}\label{conj:regular-foliation-cK}
If $\cF$ is a regular foliation, then $\cF$ is cohomologically K\"ahler.
\end{conjecture}

\begin{remark}
If we drop the condition of regularity or integrability of $\cF$, then the conclusion of \autoref{conj:regular-foliation-cK} is easily seen to fail.

Indeed, let $\X = \PP^n$ be a projective space of dimension $n \ge 2$. Then $\PP^n$ does not admit any regular foliations of dimension $0 < p < n$, as can be seen by applying Bott's vanishing theorem for characteristic classes of regular foliations.  However, $\PP^n$ admits many singular foliations of such dimensions, and also admits regular distributions of codimension one when $n$ is odd (the holomorphic contact structures).    These are never cohomologically K\"ahler: if $\cD$ is a distribution on $\PP^n$ of dimension $0 < p <n$, then $\omega_{\cD}$ is a line bundle on $\PP^n$, and hence  $\coH[p]{\PP^n,\omega_{\cD}} = 0$.
\end{remark}

\subsection{Demailly's integrability criterion}

Next, we proceed to show how Demailly's integrability criterion \cite{Demailly02}, or more precisely its proof, provides a natural class of cohomologically Kähler foliations.

We maintain the notation introduced in \autoref{subsection:cKf}. The key notion here is the following.

\begin{definition}
	Let $\mathcal D$ be a (possibly singular) distribution of codimension $q$ on a complex manifold $\M$. We say that a  closed positive current $T$ of bidegree $(q,q)$ on $\M$ is \emph{strongly directed} (with respect to $\mathcal D$) if one can locally write
		$$T={ ( \sqrt{-1})}^{q^2} a\ \omega\wedge \overline{\omega}$$
where $a$ is a positive locally finite Borel  measure and $\omega$ is a local generator of $\det\cN{\cD}^*$.
\end{definition}

\begin{lemma}\label{lemma:invariant_measure}
Let  $\Si \subseteq \X$ be an analytic subset of codimension $\geq q+2$. If $\cD$ admits a closed strongly directed positive current $T$ on $\X\setminus\Si$, then $\cD$ is cohomologically K\"ahler.	
\end{lemma}

\begin{proof}
Set $\cL=\det\cN{\cD}^*$. The closed positive current $T$ can be alternatively regarded as an $\cL$-valued $\bar\partial$-closed current on $\X\setminus\Si$. Thus it defines a cohomology class $c_{\X\setminus\Si}\in \coH[q]{\X\setminus\Si,\cL}$. Thanks to the assumption on $\Si$, the natural map $$\coH[q]{\X,\cL}\to \coH[q]{\X\setminus\Si,\cL|_{\X \setminus \Si}}$$ is actually an isomorphism (see \cite[Chapter 2]{Banica74}). Let
$c_{\X}\in \coH[q]{\X,\cL}$ be the extension of $c_{\X\setminus\Si}$.

Because $\Si$ has codimension at least $q+2 \ge q+1$, the current $T$ extends by Harvey's Theorem \cite{Harvey74} to a closed positive current $\bar T$ (non necessarily strongly directed) on $\X$. Let $\left[\bar T\right]$ be the cohomology class defined by $\bar T$ in $\coH[q]{\X, \Omega_\X^q}$.
By construction,  $\left[\bar T\right]$ is the image of the class $c_\X\in \coH[q]{\X,\mathcal L}$ under the natural map
\begin{align}
\coH[q]{\X,\cL} \to \coH[q]{\X,\forms[p]{\X}}. \label{eq:conormal-inclusion}
\end{align}

Now, consider a K\"ahler form $\kahcl$ on $\X$. We argue by contradiction and assume that $\left[\hook{v}\kahcl^p\right] = 0 \in \coH[p]{\X,\omega_{\cD}}$. By Serre duality together with the adjunction formula $\omega_{\cD}\cong \omega_{\X}\otimes \cL^*$, we have $\coH[p]{\X,\omega_{\cD}}\cong \coH[q]{\X,\omega_{\cD}^*\otimes\omega_X}^*\cong \coH[q]{\X,\cL}^*$. Hence
$\left[\kahcl^p\right]\wedge \alpha = 0$ for any class $\alpha$
in the image of the map \eqref{eq:conormal-inclusion}.  In particular,  $\left[\gamma^p\right] \wedge \left[\bar T\right]=0$. On the other hand, $\int_\X \gamma^p\wedge \bar T >0$ by the positivity of $\bar T$, a contradiction. This finishes the proof of the proposition.
\end{proof}

The following result is an easy consequence of \autoref{lemma:invariant_measure} above together with the proof of Demailly's integrability criterion.

\begin{proposition}\label{P:psefnew}
If $\det\cN{\cD}^*$ is a pseudoeffective line bundle, then ${\cD}$ is integrable and cohomologically K\"ahler.
\end{proposition}

 \begin{proof}
 Let $\cL = \det \cN{\cF}^*$.   By assumption there exists a twisted $q$-form  $\alpha\in \coH[0]{\X,\forms[q]{\X}\otimes { \mathcal{L}}^*}$ defining $\cD$. The integrability of $\cD$ follows from \cite{Demailly02}. There, it is also proved that if $h$ is a singular metric on $\cL$ with local psh weight $\varphi$ ($h$ exists by assumption), then the semi-positive  $(q,q)$-form $T= {\{\alpha,\alpha\}}_{h^*}$, is $d$-closed in the sense of currents. Here,  $\{ \cdot, \cdot \}_{h^*}$ denotes the sesquilinear pairing induced by $h^*$ on $\cL^*$-valued forms. Locally, $T$ reads $$\mu= { \sqrt{-1}}^{q^2}e^\varphi \alpha\wedge\bar{\alpha},$$ and hence, it is a closed strongly directed positive current $T$ on $\X$. One concludes by Lemma \ref{lemma:invariant_measure}.
 \end{proof}

\begin{proposition}\label{P:c1=0+psef}
If the canonical line bundle $\omega_\X$ is pseudo-effective and $c_1(\cT{\cD})=0$, then $\cD$ is regular,  integrable, and cohomologically K\"ahler.
\end{proposition}	

\begin{proof}
By \cite[Theorem 5.1]{Loray2018}, any $p$-vector $v \in \coH[0]{\X,\wedge^p\cT{\X}\otimes\omega_\cD}$ defining $\cD$ is nonvanishing.  In particular, the distribution $\cD$ is regular. The determinant
$\cL=\textrm{det}\  \cN{\cD}^*$ of the conormal bundle $\cN{ \cD}^*$
is then pseudo-effective since
$\cL \cong \omega_\X \otimes \omega_{\cD}^*$
by the adjunction formula. The claim now follows from \autoref{P:psefnew} above.
\end{proof}	

As another consequence of  \autoref{lemma:invariant_measure}, one can give a version in the K\"ahler setting of a result proved by Esteves and Kleiman in an algebraic context, see \cite{EstevesKleiman}.
\begin{proposition}
Let $\Y$ a codimension $q$ subvariety of $\X$ such that $\Y\setminus \mathrm{Sing}(\cD)$ is a ${\cD}$-invariant submanifold. Suppose that $\Si:=\textup{Sing}(\cD) \cap \Y$ has codimension $\geq 2$ in $\Y$. Then  ${\cD}$ is cohomologically K\"ahler.
\end{proposition}
\begin{proof}
This follows immediately from \autoref{lemma:invariant_measure} applied to the integration current along $\Y\setminus\Si$.
\end{proof}	

\subsection{Foliations of codimension at most two}

In this section, we confirm \autoref{conj:regular-foliation-cK} for regular foliations of codimension one as well as regular foliations of codimension two with numerically trivial canonical bundle on projective manifolds.

We maintain the notation introduced in \autoref{subsection:cKf}.

\begin{proposition}\label{prop:codimension-one-cK}
If $\cF$ has codimension $q=1$ and its singular set has codimension at least $3$, then $\cF$ is cohomologically K\"ahler.
\end{proposition}
\begin{proof}
Let $\kahcl \in \coH[1]{\X,\forms[1]{\X}}\cong\coH[1,1]{\X}$ be a K\"ahler class. We argue by contradiction and assume that $\hook{v}\kahcl^{n-1} = 0 \in \coH[n-1]{\X,\omega_{\cF}}$. By Serre duality together with the adjunction formula $\omega_{\cF}\cong \omega_{\X}\otimes \cN{\cF}$, we have
$\coH[n-1]{\X,\omega_{\cF}}\cong \coH[1]{\X,\omega_{\cF}^*\otimes\omega_X}^*\cong \coH[1]{\X,\cN{\cF}^*}^*$. Thus $\alpha \wedge \kahcl^{n-1} = 0$ for any class $\alpha$
in the image of the natural map
\[
\coH[1]{\X,\cN{\cF}^*} \to \coH[1]{\X,\forms[1]{\X}}.
\]
On the other hand, if $\alpha$ is a class as above, then $\alpha^2=0 \in \coH[2]{\X,\forms[2]{\X}}$. Together with the Hodge index theorem, this implies that the map
\[
\coH[1]{\X,\cN{\cF}^*} \to \coH[1]{\X,\forms[1]{\X}}
\]
vanishes.

Thanks to the assumption on the dimension of the singular set $\Si$ of $\cF$, the natural maps
$$\coH[1]{\X,\cN{\cF}^*}\to \coH[1]{\X\setminus\Si,\cN{\cF}^*|_{\X \setminus \Si}}$$
and
$$\coH[1]{\X,\forms[1]{\X}}\to \coH[1]{\X\setminus\Si,\forms[1]{\X\setminus \Si}}$$
are isomorphisms (see \cite[Chapter 2]{Banica74}).
By Bott's vanishing theorem applied to $\cF|_{\X\setminus \Si}$, we see that $c_1(\cN{\cF})$ lies in the image of the map

\[
\coH[1]{\X,\cN{\cF}^*} \to \coH[1]{\X,\forms[1]{\X}}.
\]
This implies $c_1(\cN{\cF})=0$, which contradicts \autoref{P:psefnew}.
\end{proof}

\begin{lemma}\label{lem:codimension-two-cK}
Suppose that $\X$ is a smooth projective variety. If $\cF$ is regular of codimension $q=2$ and $c_1(\cT{\cF})=0$, then $\cF$ is cohomologically K\"ahler.
\end{lemma}

\begin{proof}
Let $\Y$ be a minimal $\cF$-invariant analytic subspace in $\X$. Then $\cT{\cF}|_\Y$ defines a regular foliation on $\Y$ of codimension at most $2$.   If the codimension is less than two, then it is either zero or one, so $\cF|_\Y$ is cohomologically K\"ahler by \autoref{ex:extreme-cohK} or \autoref{prop:codimension-one-cK}, respectively.  Hence $\X$ is cohomologically K\"ahler by \autoref{lem:invariant-subspace-kahler}.

If the codimension of $\cF|_\Y$ is equal to two, then $\dim \Y = \dim\X$ and hence $\X = \Y$.  Thus $\X$ contains no proper invariant subspaces.
If $\omega_\X$ is pseudo-effective, then the result follows from  \autoref{P:c1=0+psef} below. Suppose from now on that $\omega_\X$ is not pseudo-effective. Applying \cite[Theorem 1.1]{DruelCodimensiontwo}, we see that one of the following holds.

\begin{enumerate}
\item There exists a $\mathbb{P}^1$-bundle structure $f \colon \X \to \Y$ over a complex projective manifold $\Y$, such that $\cF$ is everywhere transverse to $f$. In particular, $\cF$ induces a regular codimension one foliation $\cG$ on $\Y$ with $c_1(\cT{\cG})=0$.
\item There exists a smooth morphism $f\colon \X \to \Y$ onto a complex projective manifold $\Y$ of dimension $\dim \Y = \dim \X - 2$ with $c_1(\Y)=0$, and $\cF$ gives a flat holomorphic connection on $f$.
\end{enumerate}

In case (2), the result follows from \autoref{lem:direct-sum-cK}, so suppose we are in case (1). By
\autoref{prop:codimension-one-cK}, $\cG$ is cohomologically K\"ahler, and hence by \autoref{prop:main_lemma} below, there exists a foliation $\cL$ of dimension $1$ on $\Y$ such that $\cT{\X}\cong \cT{\cG}\oplus \cT{\cL}$. But then $\cT{\X}$ decomposes as a direct sum $\cT{\X} \cong \cT{\cF} \oplus \cT{f^{-1}\cL}$. Hence the claim follows from \autoref{lem:direct-sum-cK} again.
\end{proof}

\subsection{Splitting the tangent bundle when $c_1=0$}
We now describe the behaviour of cohomologically K\"ahler foliations whose first Chern class is trivial.

Let $\X$ be a compact K\"ahler manifold, and let $\cF$ be a possibly singular foliation on $\X$ of dimension $p$ with $c_1(\cT{\cF})=0$.  Then $\omega_{\cF}$ is a Hermitian flat line bundle.  Hence if $\kahcl$ is any K\"ahler form on $\X$, and $v \in \coH[0]{\X,\wedge^p\cT{\X}\otimes \omega_{\cF}}$ is a $p$-vector defining $\cF$, the class $\left[\hook{v}\gamma^p\right] \in \coH[p]{\X,\omega_{\cF}}$ has an Hermitian conjugate
\[
\alpha := \overline{\left[\hook{v}\gamma^p\right]} \in \coH[0]{\X,\forms[p]{\X}\otimes \omega_\cF^*},
\]
by Hodge symmetry for unitary local systems.
\begin{lemma}\label{lem:nonzero-contraction}
If $\cF$ is weakly cohomologically K\"ahler with respect to $\left[\gamma\right]\in \coH[1,1]{\X}$, then the contraction
\[
\hook{v}\alpha\in \coH[0]{\X,\cO{\X}}
\]
is nonvanishing.
\end{lemma}
\begin{proof}
The proof is similar to the arguments in \cite[Lemma 5.25]{Druel2020}.
From the definition of $\alpha$, we have $\overline{\alpha} = \hook{v}\gamma^p + \overline\partial \xi$ for some $\xi$ where $\overline{\alpha}$ is the complex conjugate of $\alpha$ with respect to the Hermitian flat structure.  A straightforward calculation using Stokes' formula and Poincar\'e duality for the unitary local system $\omega_\cF$ then implies that
\[
\int_{\X} \alpha\wedge\overline{\alpha}\wedge \kahcl^{n-p} = \int_\X \alpha\wedge \hook{v}\gamma^p \wedge \kahcl^{n-p}.
\]
The integral on the left is nonzero by the Hodge--Riemann bilinear relations, while an easy pointwise calculation shows that the integrand on the right is given by
\[
\alpha\wedge\hook{v}\kahcl^p\wedge \kahcl^{n-p} = \frac{1}{{n \choose p}}\cdot \hook{v}\alpha\cdot\kahcl^n
\]
Note that $\hook{v}\alpha$ is a holomorphic function, hence constant.  We therefore have
\[
0 \neq  {n\choose p} \int_\X \alpha\wedge\hook{v}\kahcl^p\wedge \kahcl^{n-p} = \hook{v}\alpha \cdot \int_\X \kahcl^n
\]
and we conclude that $\hook{v}\alpha$ is nonzero, hence nonvanishing, as desired.
\end{proof}

The following result is a generalization of \cite[Theorem 5.6]{Loray2018}; see also \cite[Proposition 3.1]{Druel2019} for a related result.
\begin{proposition}\label{prop:main_lemma}
Let $\X$ be a compact K\"{a}hler manifold, and let $\cF$ be a possibly singular foliation on $\X$ of dimension $p$  with $c_1(\cT{\cF})=0$. If $\cF$ is  weakly cohomologically K\"ahler,
then the following statements hold:
\begin{enumerate}
\item $\cF$ is regular,
\item There exists a foliation $\cG$ on $\X$ such that $\cT{\X}=\cT{\cF}\oplus\cT{\cG}$.
\item $\det \cT{\cF}$ is a torsion line bundle.
\end{enumerate}
\end{proposition}

\begin{proof}
For the first statement, note that by \autoref{lem:nonzero-contraction}, any $p$-vector defining $\cF$ is nonvanishing, and hence $\cF$ is regular.

For the second statement, note that contraction with $\alpha$ gives a morphism $\wedge^{p-1} \cT{\cF} \to \Omega_\X^1\otimes\omega_\cF^*$ such that the composition
$$\wedge^{p-1}\cT{\cF} \to \forms[1]{\X} \otimes\omega_\cF^* \to \forms[1]{\cF}\otimes\omega_\cF^*$$
is an isomorphism. The kernel of the induced map $\cT{\X} \to \forms[p-1]{\cF}\otimes\omega_\cF^*$ then defines a distribution $\cG$ such that
$\cT{\X} \cong \cT{\cF} \oplus \cT{\cG}$.  Let $\beta \in \coH[0]{\X,\Omega_\X^p\otimes\omega_\cF^*}$ be a twisted $p$-form defining $\cG$. Using the K\" ahler identities, we see that $\beta$ is closed with respect to any unitary flat connection on $\omega_\cF^*$. This easily implies that $\cG$ is involutive.

For the third statement, we follow the argument in the proof of \cite[Theorem 5.2]{Loray2018}; we detail the proof for the sake of completeness. It relies on properties of cohomology jump loci in the space of rank one local systems proved by Simpson \cite{Simpson93} for projective manifolds and by Wang \cite{Wang16} for compact K\"ahler spaces.

Replacing $\X$ by a finite \'etale cover, if necessary, we may assume without loss of generality that
$\omega_\cF \in \mathsf{Pic}^0 (\X)$. Also, recall that $\coH[0]{
\X, \Omega^{p}_\X \otimes \omega_\cF^*} \cong \coH[p]{\X, \omega_\cF}$ by Hodge symmetry with coefficients in unitary local systems.
Let $m := \coh[p]{\X, \omega_\cF}$, and consider the Green--Lazarsfeld set
	\[
	S = \{ [\mathcal L]  \in  \mathsf{Pic}^0 (\X)\, | \, \coh[p]{\X, \mathcal L} \ge m \}\ni [\omega_\cF].
	\]
Then by \cite[Theorem 1.3, Corollary 1.4]{Wang16},  $S$ is a finite union of translates of subtori by torsion points. Therefore, to prove that $\omega_\cF$ is torsion, it suffices to show that $[\omega_\cF]$ is an isolated point of $S$. Let $\Sigma \subset  \mathsf{ Pic}^0 (\X)$ be an  irreducible component of $S$ passing through $[\omega_\cF]$. Let $\mathcal P$ denote the restriction of the Poincar\'e line bundle to  $\Sigma \times \X$, and let $\pi: \Sigma \times \X \to \Sigma$ denote the projection morphism. Recall that $\mathsf{R}^p \pi_* \mathcal P$ is locally free on some open neighborhood of $[\omega_\cF]$ in $\Sigma$. As a consequence, we can extend the class in  $\coH [p]{\X,\omega_\cF}$ corresponding to $\gamma$ to a holomorphic family of nonzero cohomology classes with coefficients in line bundles $\mathcal L$ with $[\mathcal L] \in \Sigma$ sufficiently close to $[\omega_\cF]$. Then Hodge symmetry (with coefficients in local systems) gives us a family of holomorphic $p$-forms with coefficients in the dual bundles ${\mathcal L}^*$. Taking the wedge product  of these twisted $p$-forms with a twisted $(n-p)$-form defining $\mathcal F$, we obtain a  global section of $\omega_\X \otimes \det \cN{\cF} \otimes {\mathcal L}^* \cong \omega_{\cF}\otimes \mathcal{L}^*$ for any line bundle $\mathcal L$ with $[\mathcal L] \in \Sigma$ close enough to $[\omega_\cF]$,   which is nonzero since $\cG$ is everywhere transverse to $\cF$. Therefore $\mathcal L \cong \omega_\cF$ is $[\mathcal L]\in \Sigma$ is sufficiently close to $[\omega_\cF]$, as desired.
\end{proof}

Note that \autoref{conj:regular-foliation-cK} and \autoref{prop:main_lemma} together imply the following.

\begin{conjecture}
Let $\X$ be a compact K\"{a}hler manifold, and let $\cF$ be a regular foliation on $\X$ with $c_1(\cT{\cF})=0$. Then the following statements hold.
\begin{enumerate}
\item There exists a foliation $\cG$ on $\X$ such that  $\cT{\X} = \cT{\cF} \oplus \cT{\cG}$; and
\item $\det \cT{\cF}$ is a torsion line bundle.
\end{enumerate}
\end{conjecture}

Note also that Beauville has conjectured that a splitting of the tangent bundle into complementary foliations is induced by a splitting of the universal cover as a product~\cite{Beauville2000}.  Combining that conjecture with the above, we arrive at the following.
\begin{conjecture}\label{conj:CY-Beauville}
Let $\X$ be a compact K\"ahler manifold, and let $\cF$ be a regular foliation of $\X$ with $c_1(\cT{\cF})=0$.  Then there exist possibly non-compact K\"ahler manifolds $\Y$ and $\Z$, with $\omega_\Y$ trivial, and a covering map $f : \Y \times \Z \to \X$ such that $f^{-1}{\cF}$ is induced by the projection to $\Z$.
\end{conjecture}

\section{Foliations with numerically flat tangent bundle}\label{S:proof main}

This section is mostly taken up by the proof of \autoref{thm:num_flat_main} below. Its  statement  is almost the same as the statement of \autoref{thm:num_flat_new_main} from the Introduction. The only difference is the inclusion of an extra item describing the analytic closure of minimal $\cF$-invariant analytic subspaces. Before presenting the theorem, we introduce some terminology. A compact complex manifold $\X$ is called a \emph{torus quotient} if $\X$ is a finite \'etale quotient of a complex torus $\T$. A (regular) foliation $\cF$ on a torus quotient $\X$ is said to be \emph{linear} if $f^{-1}\cF$ is a linear foliation on $\T$, where $f \colon \T \to \X$ denote the quotient map.

\begin{theorem}\label{thm:num_flat_main}
    Let $\X$ be a compact K\"{a}hler manifold, and let $\cF$ be regular foliation on $\X$ with numerically flat tangent bundle $\cT{\cF}$. Then the following hold.
    \begin{enumerate}
        \item\label{I:hermitian flat} The tangent bundle $\cT{\cF}$ is hermitian flat.
        \item \label{I:coh-kah} $\cF$ is cohomologically K\"ahler.
        \item\label{I:abundance} The line bundle $\omega_{\cF}$ is torsion.
        \item\label{I:split} There exists a foliation $\cG$ on $\X$ such that $\cT{\X} = \cT{\cF} \oplus \cT{\cG}$.
        \item\label{I:Beauville} $\cF$ is induced by a splitting of the universal cover $\widetilde{\X}$, i.e.~ $\widetilde{\X}$ is the product of $\mathbb{C}^{\dim \cF}$ with another complex manifold $\Y$ in such a way that the decomposition $\cT{\X}=\cT{\cF}\oplus \cT{\cG}$ lifts to the canonical decomposition $\cT{\widetilde{\X}}\cong
    \cT{\mathbb{C}^{\dim \cF}} \boxplus \cT{\Y}$.
        \item\label{I:minimal} Any minimal $\cF$-invariant closed analytic subspace $\Y$ is a torus quotient and $\cF|_\Y$ is a linear foliation.
        \item\label{I:closure} The analytic closure $\overline{\bL}$ of any leaf $\bL$ of $\cF$ is a quotient of an equivariant compactification of an abelian Lie group $\G_{\bL}$ and $\cF|_{\overline{\bL}}$ is induced by a (not necessarily closed) subgroup of $\G_{\bL}$.
    \end{enumerate}
\end{theorem}

The rest of this section is devoted to the proof of this theorem.  Hence for the rest of this section we adopt the following assumptions:
\begin{itemize}
\item $\X$ is a compact K\"ahler manifold
\item $\cF$ is a regular foliation
\item The dimension $p = \dim \cF$ is positive
\item $\Y \subseteq \X$ is a minimal $\cF$-invariant closed analytic subspace.
\end{itemize}
Here, by minimal, we mean ``minimal with respect to inclusion''.   Note that since  $\cF$ is regular, the singular locus of an invariant subspace is also invariant, and hence $\Y$ is automatically smooth.

For clarity, we break the proof of \autoref{thm:num_flat_main} into several lemmas which address various implications.   We first explain how the first statement (that $\cT{\cF}$ is hermitian flat) can be used to deduce the others, and later show that this first statement does indeed hold (\autoref{L:numerically_flat_versus_hermitian_flat}).

\begin{lemma}[\ref{I:hermitian flat} $\implies $ \ref{I:minimal}]\label{L:minimal_invariant_hermitian_flat}
    If $\cT{\cF}$ is hermitian flat, then any  minimal $\cF$-invariant closed analytic subspace $\Y$ is a torus quotient, and $\cF|_\Y$ is a linear foliation.
\end{lemma}
\begin{proof}
    By assumption, $\cT{\cF}$  is given by a unitary representation
        $$\rho\colon \pi_1(\X) \to \U(p).$$
    In particular, $\cT{\cF}|_\Y$ is hermitian flat as well. Replacing $\X$ by $\Y$, we may therefore assume without loss of generality that there is no proper minimal $\cF$-invariant analytic subspace in $\X$.  We now break the proof into several cases.

    \medskip

    \textbf{Case 1: $\cT{\F}$ is trivial.}  Suppose that $\cT{\cF}\cong \cO{\X}^{\oplus p}$.  We claim that $\coH[0]{\X,\cT{\cF}}\subseteq \coH[0]{\X,\cT{\X}}$ is an abelian Lie subalgebra. Indeed, let $v_i \in \coH[0]{\X,\cT{\cF}}$ for $i\in\{1,2\}$. Then $[v_1,v_2]$ induces the zero flow on the Albanese torus $\A$ of $\X$. Thus, by \cite[Theorem 3.14]{Lieberman78}, the vector field $[v_1,v_2]$ has at least one zero. On the other hand, $[v_1,v_2] \in \coH[0]{\X,\cT{\cF}}$ since $\cT{\cF}$ is involutive. Since any vector field tangent to $\cF$ has empty vanishing locus by assumption, we must have $[v_1,v_2]=0$, proving our claim.

    Let $\mathsf{H} \subseteq \Aut{\X}_0$ be the analytic closure of the complex Lie subgroup exponentiating the Lie algebra $\coH[0]{\X,\cT{\cF}}$. Note that $\mathsf{H}$ is an abelian complex Lie group. By \cite[Theorems 3.3, 3.12 and 3.14]{Lieberman78}, there is an exact sequence
    \[
        \begin{tikzcd}
        1 \arrow[r] & \N \arrow[r] & \mathsf{H} \arrow[r] & \T \arrow[r] & 1
        \end{tikzcd}\label{eq:lieberman-exact}
    \]
    where $\N$ a closed subgroup whose Lie algebra is contained in the Lie algebra of holomorphic vector fields with nonempty zero locus, and $\T$ is a compact complex torus. Moreover, $\N$ is a commutative linear algebraic group.
    By \cite[Proposition, p. 53]{Sommese73}, the neutral component $\N_0$ of $\N$ has a fixed point $x$ on $\X$ whose stabilizer in $\mathsf{H}$ is denoted by $\mathsf{H}_x$. Then
    $\mathsf{H}\cdot x$ is a compact analytic subvariety which is invariant under $\cF$, and hence
    $\mathsf{H}\cdot x = \X$. This immediately implies that $\X$ is the quotient of the torus $\mathsf{ H}/\N_0$ by the finite group $\mathsf{ H}_x/\N_0$, and that $\cF$ is a linear foliation on $\X$.

    \medskip

\textbf{Case 2: $\cF$ is one-dimensional.}  Let $a\colon \X \to \A$ be the Albanese morphism, and let $q(\X) = \dim \A$ be the irregularity of $\X$.  If $q(\X) = 0$, then $\cT{\cF}$ is a torsion line bundle, and by passing to a finite cover we reduce to the case treated in Step 1.

    Thus suppose $q(\X)>0$. If $\cT{\cF}$ is tangent to any fiber of $a$, then this fiber is a proper $\cF$-invariant subvariety contained in $\X$, yielding a contradiction. Therefore, the composition $\cT{\cF} \to \cT{\X} \to a^*\A$ is nonvanishing, and hence $\cT{\cF}\cong \cO{\X}$. So by Step 1 again, we deduce that $\X$ is a torus quotient, and that $\cF$ a linear foliation.

\medskip

\textbf{Intermezzo.} To formulate the remaining cases, we need to introduce some additional notation and simplifications. Let $\G \subseteq \textup{GL}(p,\mathbb{C})$ be the Zariski closure of $\rho(\pi_1(\X))$. This is a linear algebraic group which has finitely many connected components. Applying Selberg's Lemma and passing to an appropriate finite \'etale cover of $\X$, we may assume without loss of generality that $\G$ is connected, and that the image of the induced representation
        $$\rho_1 \colon \pi_1(\X) \to \G \to \G/\textup{Rad}(\G)$$
    is torsion-free, where $\textup{Rad}(\G)$ denotes the radical of $\G$.

\medskip

\textbf{Case 3: $\rho_1(\pi_1(\X))$ is infinite}.  We claim that this case cannot occur.  To see this, let
        $$f\colon \X \to \Z$$
be the $\rho_1$-Shafarevich morphism. We refer to \cite[Definition 2.13]{CCE15} for this notion and to \cite[Th\'eor\`eme 1]{CCE15} for its existence. Moreover (\textit{loc. cit.}), as we are assuming that  $\rho_1(\pi_1(\X))$ is torsion free, we may assume that $\Z$ is a normal projective variety of general type. Note that $\dim \Z>0$ since $\rho_1(\pi_1(\X))$ is infinite.

    Observe that $\cF$ is not tangent to the fibers of $f$ since there is no proper $\cF$-invariant analytic subspace in $\X$. Note further that flat sections of the unitary vector bundle $\cT{\cF}$ lift to holomorphic vector fields on the universal covering of $\X$ that have bounded norm, and are therefore complete. The orbits of these vector fields give rise by projection to entire curves covering  $\X$ as well as $\Z$.  But according to \cite[Theorem 7.4.7]{Kobayashi98} varieties of general type cannot be covered by entire curves, a contradiction.  Hence $\rho_1(\pi_1(\X))$ cannot be infinite.

\medskip
    \textbf{Case 4: $\rho_1(\pi_1(\X))$ is finite.}
    If $\rho_1(\pi_1(X))$ finite, then it is trivial, since it is torsion-free by construction. It follows that $\G$ is solvable. Since $\rho$ is unitary and $\G$ is connected, we deduce that $\rho$ splits as a direct sum of rank one representations. Thus $\cT\cF\cong \cF_1 \oplus \cdots \oplus \cF_p$ where $\dim \cF_i=1$, and $c_1({\cF_i})=0$.

     We now proceed by induction on $\dim \X$.   If $\dim \X = 1$, then $\cT{\cF}=\cT{\X}$ and hence $c_1(\X)=c_1(\cF)=0$, so that $\X$ is an elliptic curve and the statement holds.

     Now suppose that $\dim\X > 1$.  If there exists $i$ such that $\cF_i$ has no proper $\cF_i$-invariant subspace, then by Step 2, $\X$ is a torus quotient and $\cF_i$ is a linear foliation.  The remaining foliations $\cF_j$ are then also linear: indeed, since the bundles $\cT{\cF_j}$ are flat, their first Chern classes are zero and hence they must correspond to trivial summands of the trivial bundle $\cT{\X}$.

Otherwise, $\cF_i$ has an invariant subspace $\Y_i$ for every $i$.  By the induction hypothesis, $\Y_i$ is a torus quotient, and hence $\cF_i|_{\Y_i}$ is cohomologically K\"ahler by \autoref{ex:torus-quotient-kahler}, so that $\cF_i$ itself is cohomologically K\"ahler by \autoref{lem:invariant-subspace-kahler} and thus $\cT{\cF_i}=\det\cT{\cF_i}$ is a torsion line bundle by \autoref{prop:main_lemma}.  Hence by passing to a finite \'etale cover, we may assume that $\cT{\cF_i}$ is trivial for all $i$, and hence $\cT{\cF}$ itself is trivial.  The result therefore follows from Case 1.
\end{proof}

Since every linear foliation on a torus quotient is cohomologically K\"ahler (\autoref{ex:torus-quotient-kahler}), and it suffices to check the latter condition on submanifolds (\autoref{lem:invariant-subspace-kahler}), we deduce the following.
\begin{corollary}[\ref{I:hermitian flat} $\implies$ \ref{I:coh-kah}]
If $\cT{\cF}$ is hermitian flat, then $\cF$ is cohomologically K\"ahler.
\end{corollary}

Hence by \autoref{prop:main_lemma} we have the following.
\begin{corollary}[\ref{I:hermitian flat} $\implies$ \ref{I:abundance} and \ref{I:split}] \label{C:existence_complementary_hermitian_flat}
If $\cT{\cF}$ is hermitian flat, then $\omega_\cF$ is a torsion line bundle, and there exists a regular foliation $\cG$ on $\X$ such that
    $\cT{\X} = \cT{\cF} \oplus \cT{\cG}$.
\end{corollary}

\begin{lemma}[\ref{I:hermitian flat} $\implies$ \ref{I:Beauville}]
If $\cT{\cF}$ is Hermitian flat, then $\cF$ is induced by a splitting of the universal cover $\widetilde{\X}$.
\end{lemma}

\begin{proof}
This is argued in the proof of \cite[Theorem A]{Pereiratouzet2013}.
\end{proof}

\begin{lemma}[\ref{I:hermitian flat} $\implies$ \ref{I:closure}]
If $\cT{\cF}$ is hermitian flat, then the leaf closures are induced by abelian group actions as in part  \ref{I:closure} of \autoref{thm:num_flat_main}.
\end{lemma}
\begin{proof}
In the proof of \autoref{L:minimal_invariant_hermitian_flat}, we saw that the leaves of $\cF$ are contained in the fibers of the Shafarevich morphism. Furthermore, the restriction of $\cF$ to a fibre of the Shafarevich morphism is defined by an analytic action of $\mathbb{C}^{\dim \cF}$. Therefore, the analytic closure $\overline{\bL}$ of any leaf $\bL$ admits a locally free $\mathbb{C}^{\dim \cF}$-action (i.e. the stabilizer of any point is discrete) with an analytically dense orbit.
To conclude, it suffices to take $\G_{\bL}$ equal to the analytic closure of the subgroup of $\Aut{\overline \bL}_0$ determined by the locally free $\mathbb{C}^{\dim \cF}$-action defining $\cF_{|\overline \bL}$.
\end{proof}

At this point, we have shown that all statements in \autoref{thm:num_flat_main} follow from the first.  It thus remains to establish the following:

\begin{lemma}[\ref{I:hermitian flat} holds] \label{L:numerically_flat_versus_hermitian_flat} If $\cT{\cF}$ is numerically flat, then it is hermitian flat.
\end{lemma}
\begin{proof}Let $\cT{1} \subseteq \cT{\cF}$ be an hermitian flat subbundle of maximal rank $p_1$. By Theorem \ref{thm:numerically_flat}, we have $p_1 \ge 1$.  We must show that $\cT{1}=\cT{\cF}$.  We do so by treating several cases, similar to the proof of \autoref{L:minimal_invariant_hermitian_flat},  as follows.
    \medskip

    \textbf{Case 1: $\cT{1}$ is involutive.} In this case, by Corollary \ref{C:existence_complementary_hermitian_flat} applied to $\cT{1}$, there exists a foliation $\cG$ such that $\cT{\X}\cong \cT{1}\oplus \cT{\cG}$. Then $\cF \cap \cG \subseteq \cF$ gives a splitting $\cT\cF \cong \cT{1} \oplus \cT{ \cF \cap \cG}$. But then $\cF \cap \cG$ is numerically flat as well and applying Theorem \ref{thm:numerically_flat} to $\cT{ \cF \cap \cG}$ contradicts the maximality of $\rank \cT{1}$ unless $\cT{1}=\cT{\cF}$.

    \textbf{Case 2: $\cT{1}\cong \cO{\X}^{\oplus p_1}$ is trivial.}  We claim that $\coH[0]{\X,{\cT{1}}}\subseteq \coH[0]{\X,\cT{\X}}$ is an abelian Lie subalgebra. Indeed, let $v_i \in \coH[0]{\X,{\cT{1}}}$ for $i\in\{1,2\}$. Then $[v_1,v_2]$ induces the zero flow on the Albanese torus $\A$ of $\X$. Thus, by \cite[Theorem 3.14]{Lieberman78}, the vector field $[v_1,v_2]$ has at least one zero. On the other hand, $[v_1,v_2] \in \coH[0]{\X,\cT{\cF}}$ since $\cF$ is a foliation.

    Applying \cite[Proposition 1.16]{DPS94}, we see that $[v_1,v_2]=0$, proving our claim. In particular, $\cT{1}$ is involutive, and the conclusion follows from Step 1.

\medskip

    \textbf{Intermezzo:} Suppose now that the vector bundle $\cT{1}$ is given by an aribtrary unitary representation
        $$\rho\colon \pi_1(X) \to \U(p_1).$$
    Let $\G \subseteq \textup{GL}(p_1,\mathbb{C})$ be the Zariski closure of $\rho(\pi_1(X))$.
    As before, after passing to an appropriate finite \'etale cover of $\X$, we may assume without loss of generality that $\G$ is connected, and that the image of the induced representation
        $$\rho_1 \colon \pi_1(X) \to \G \to \G/\textup{Rad}(\G)$$
    is torsion free.

    \medskip

    \textbf{Case 3: $\rho_1(\pi_1(X))$ is finite.} Then $\rho_1(\pi_1(X))$ is the trivial group since it is torsion free by construction. It follows that $\G$ is solvable. Since $\rho$ is unitary, we deduce that $\rho$ splits as a direct sum of rank one representations. By  \autoref{C:existence_complementary_hermitian_flat}, we deduce that $\rho(\pi_1(X))$ is finite as well. The conclusion now follows from Case 2 after passing to an \'etale cover.

    \medskip

    \textbf{Case 4:  $\rho_1(\pi_1(X))$ is infinite.} Let
    $$f\colon \X \to \Z$$
    be the $\rho_1$-Shafarevich map. By \cite[Th\'eor\`eme 1]{CCE15} again, we may assume without loss of generality that $\Z$ is a (positive dimensional) normal projective variety of general type. Arguing as in the end of the proof of  \autoref{L:minimal_invariant_hermitian_flat}, we see that $\cT{1}$ is tangent to the fibers of $f$. By induction on $\dim \X$, we conclude that $\cT{1}$ is involutive. The conclusion now follows from Case 1, proving the lemma.
\end{proof}

\begin{remarkx}
  One can naturally  wonder if the tangent bundle of a foliation satisfying the hypotheses of \autoref{thm:num_flat_main} becomes analytically trivial after some suitable finite \'etale cover. It turns out that the answer is negative in general. An example of a rank two foliation such that the above representation $$\rho\colon \pi_1(\X) \to \U(2)$$ has infinite image is given in \cite[Section 4]{Pereiratouzet2013}.
\end{remarkx}

The following consequence of \autoref{lem:invariant-subspace-kahler}, \autoref{prop:main_lemma}, \autoref{P:c1=0+psef} and \autoref{thm:num_flat_main} will be useful below and may be of independent interest.

\begin{proposition}\label{existence_complementary_hermitian_flat_2}
Let $\X$ be a compact K\"{a}hler manifold, and let $\cF$ be a possibly singular foliation on $\X$ with $c_1(\cT{\cF})=0$. Suppose that there exists an $\cF$-invariant compact submanifold $\Y$ entirely contained in the regular locus of $\cF$ such that $\cT{\cF}|_\Y$ is numerically flat.
Then $\cF$ is regular, $\omega_\cF$ is a torsion line bundle, and there exists a foliation $\cG$ on $\X$ such that $\cT{\X} = \cT{\cF} \oplus \cT{\cG}$.
\end{proposition}

\section{Foliations of dimension two with numerically trivial canonical bundle}
\label{sec:dim2-CY}

In this section, we describe the structure of regular foliations with zero first Chern class and dimension at most two on complex projective manifolds (see Theorem \ref{theorem:classification_small_rank}).

Notions of singularities coming from the minimal model program have been shown to be very useful when studying (birational) geometry of foliations. We refer the reader to \cite[Section I]{McQuillan08} for their precise definition. The proof of  \autoref{theorem:classification_small_rank} below makes use of the following result, which might be of independent interest.

\begin{theorem}\label{abundance_foliation_by_cuvres}
Let $\X$ be a compact K\"{a}hler manifold and let $\mathcal{L}$ be a foliation by curves on $\X$ such that $c_1(\omega_{\mathcal{L}})=0$. If $\mathcal{L}$ has log canonical singularities, then $\omega_{\mathcal{L}}$ is a torsion line bundle.
\end{theorem}

\begin{proof}
Let $a\colon \X \to \A$ be the Albanese morphism, and set $\Y:=a(\X)$. If $q(\X)=0$, any line bundle whose first Chern class is zero in $\coH[2]{\X,\QQ}$ is torsion. We may thus assume $q(\X)=\dim \A>0$, and that $\mathcal{L}$ is tangent to the fibers of the induced map $\X \to \Y$.

Suppose first that the singular locus $\Z$ of $\mathcal{L}$ maps onto a proper subset of $\Y$. Then a general fiber $\F$ of the map $\X \to \Y$ is $\mathcal{L}$-invariant and $\mathcal{L}|_{\F} \subseteq \cT{\F}$ is regular.  \autoref{existence_complementary_hermitian_flat_2} then implies that
$\omega_{\mathcal{L}}$ is a torsion line bundle.

Suppose from now on that $\Z$ maps onto $\Y$, and let $\Z_1$ be any irreducible component of $\Z$ mapping onto $\Y$.
There exists a positive integer $m$ such that $$\omega_{\mathcal{L}}^{\otimes m} \in a^*\textup{Pic}(\A).$$ Hence, it suffices to show that $\omega_{\mathcal{L}}|_{\Z_1}$ is a torsion line bundle. We now argue as in \cite[Section 4.1]{BM16}. By the very definition of $\Z$, the natural map
$\cT{\mathcal{L}} \to  \cT{\X}$ gives a surjective map $$\Omega_\X^1 \twoheadrightarrow \mathcal{I}_{\Z/\X}\otimes\omega_{\mathcal{L}}.$$
Next, consider the composition
\begin{multline*}
\Lambda\colon \Omega_\X^1|_{\Z_1} \to (\mathcal{I}_{\Z/\X}\otimes\omega_{\mathcal L})|_{\Z_1} \to
(\mathcal{I}_{\Z_1/\X}\otimes\omega_{\mathcal L})|_{\Z_1}\\ \cong (\mathcal{I}_{\Z_1/\X}/\mathcal{I}_{\Z_1/\X}^2)\otimes\omega_{\mathcal L}|_{\Z_1} \to \Omega_\X^1|_{\Z_1} \otimes\omega_{\mathcal L}|_{\Z_1}.
\end{multline*}
If $z \in \Z_1$ is any point and $v$ is a local generator of $\mathcal{L}$ on some open neighborhood of $z$, then the induced map
$$\Lambda|_z\colon \Omega_\X^1|_z \to \Omega_\X^1|_z \otimes\omega_{\mathcal L}|_z \cong \Omega_\X^1|_z$$
is the linear part of $v$ at $z$.
By \cite[Facts I.1.8 and I.1.9]{McQuillan08}, the endomorphism $\Lambda|_z$ of $\Omega_\X^1|_z$ is not nilpotent since $\mathcal{L}$ has log canonical singularities by assumption. It follows that there exists $k \in \{1,\ldots,\dim \X\}$ such that one of the $k^\textup{th}$ elementary symmetric function
$s_k \in \coH[0]{\Z_1, { ( { \omega_{\mathcal L}} |_{\Z_1})}^{\otimes k}}$
of $\Lambda \in {\textup{End}}_{\mathcal{O}_{\Z_1}}(\Omega_{\X}^1|_{\Z_1})\otimes \omega_{\mathcal L}|_{\Z_1}$
does not vanish at $z$ and in particular $s_k$ is a nonzero section of $\omega_{\mathcal{L}}|_{\Z_1}^{\otimes k}$.  Since $c_1(\omega_{\mathcal{L}})=0$, this implies that $\omega_{\mathcal{L}}|_{\Z_1}$ is torsion, as desired.
\end{proof}

\begin{theorem}\label{theorem:classification_small_rank}
Let $\X$ be a complex projective manifold, and let $\cF$ be a regular foliation of dimension $p \in\{1,2\}$ with $c_1(\cT{\cF})=0$. Then either $\cF$ is algebraically integrable, or $\cT{\cF}$ is numerically flat. Moreover, $\omega_\cF$ is a torsion line bundle, and there exists a regular foliation $\cG$ on $\X$ such that $\cT{\X} = \cT{\cF} \oplus \cT{\cG}$.
\end{theorem}
\begin{proof}
The second assertion follows from \cite[Theorem 5.6]{Loray2018} if $\cF$ is algebraically integrable and from \autoref{thm:num_flat_main} if $\cT{\cF}$ is numerically flat, so we need only prove the  first assertion. If $p=1$, $\cT{\cF}$ is numerically flat by assumption. We may thus assume that $p=2$.

Let $\xi$ be the tautological class on $\mathbb{P}(\forms[1]{\cF})$.  If $\xi$ is not pseudo-effective, then $\cF$ is algebraically integrable by \cite[Proposition 8.4]{Druel18}, so suppose from now on that $\xi$ is pseudo-effective.

Let $\mathsf{H}$ be an ample divisor on $\X$. From \cite[Theorem 4.7]{CP19}, we deduce that $\cT{\cF}$ is $\mathsf{H}$-semistable. Applying \cite[Theorem IV.4.8]{Nakayama04}, we see that one of the following holds.
\begin{enumerate}
\item The tangent bundle $\cT{\cF}$ is numerically flat.
\item There exist line bundles $\mathcal{L}$ and $\mathcal{M}$ with $c_1(\mathcal{L})=-c_1(\mathcal{M})$ and slopes $\mu_\mathsf{H}(\mathcal{L})=\mu_\mathsf{H}(\mathcal{M})=0$ such that $\cT{\cF}\cong \mathcal{L}\oplus \mathcal{M}$.
\item There exist a finite \'etale cover $f \colon \X' \to \X$ of degree $2$ and a line bundle $\mathcal{L}$ on $\X'$ such that $\cT{\cF}\cong f_* \mathcal{L}$.
\item There exists a foliation $\mathcal{L}$ by curves, defined by a saturated line bundle $\cT{\mathcal{L}} \subset \cT{\cF}$ with $c_1(\cT{\mathcal{L}})=0$.
\end{enumerate}
Hence it suffices to treat cases 2 through 4, which we do as follows.

\textbf{Case 2:} In his case, \cite[Theorem 1.2]{CP19} applies to show that $\mathcal{L}^*$ and $\mathcal{M}^*$ are pseudo-effective line bundles. This immediately implies $c_1(\mathcal{L})=c_1(\mathcal{M})=0$ since $c_1(\mathcal{L})=-c_1(\mathcal{M})$ by our current assumption.
As a consequence, $\cT{\cF}$ is numerically flat.

\textbf{Case 3:} Let $\tau$ be the involution of the covering $f\colon \X' \to \X$. By
\cite[Lemma 4.11]{Nakayama04}, we have $f^*\cT{\cF}\cong \mathcal{L}\oplus \tau^*\mathcal{L}$.
Note that $f^{-1}\cF$ is a regular foliation on $\X'$ with $\cT{f^{-1}\cF}=f^*\cT{\cF}$. In particular,
$c_1(\cT{f^{-1}\cF})=0$. Then $\mathcal{L}$ is pseudo-effective by \cite[Theorem 1.2]{CP19}, and hence
$c_1(\mathcal{L})=c_1(\tau^*\mathcal{L})=0$ since $c_1(\cT{f^{-1}\cF})=0$.  This implies that $\cT{\cF}$ is numerically flat in this case as well.

\textbf{Case 4:} This one is the most involved and occupies the rest of the proof.  In this case, $\cF$ has canonical singularities
by \cite[Lemma 3.10]{fano_fols}. By \cite[Corollary 3.8]{Loray2018}, $\cF$ is not uniruled. It follows that $\mathcal{L}$ is not uniruled as well, and hence it has canonical singularities using \cite[Corollary 3.8]{Loray2018} again. Therefore $\omega_{\mathcal{L}}$ is a torsion line bundle by   \autoref{abundance_foliation_by_cuvres}. Hence, replacing $\X$ by a finite \'etale cover if necessary, we may assume without loss of generality that there is a nonzero global vector field $v \in \coH[0]{\X,\cT{\cF}}$ such that $\cT{\mathcal{L}}=\mathcal{O}_\X \, v$.

If $v$ is nowhere vanishing, i.e $\mathcal L$ is a regular foliation, then there exists a foliation $\cG$ on $\X$ everywhere transverse to $\mathcal L$ by \cite[Theorem 3.14]{Lieberman78}. The foliation $\cG$
gives a splitting $\cT\cF=\cT{ \mathcal L}\oplus \cT{\cF\cap \cG}$, and hence $\cT\cF$ is numerically flat.

Let $\Y \subseteq \X$ be a minimal $\cF$-invariant analytic subspace, which is automatically smooth. Either $v|_\Y$ is nowhere vanishing, nonzero but vanishing, or identically zero.

Suppose first that $v|_\Y$ is nowhere vanishing.  Then \autoref{existence_complementary_hermitian_flat_2} applied to $\mathcal{L}$
 implies that $\mathcal{L}$ is regular, so that $\cT{\cF}$ is numerically flat.

Suppose next that $v|_\Y \neq 0$ but $v(y)=0$ for some point $y \in \Y$. Let $\G$ be the Zariski closure in $\Aut{\Y}_0$ of the complex Lie group exponentiating $v|_\Y$. Notice that $\G$ is a commutative linear algebraic group by \cite[Theorems 3.12 and 3.14]{Lieberman78} and that $\G$ preserves ${\cF}|_\Y$. Then $\G$ fixes $y$, and thus the leaf $\bL$ of
${\cF}|_\Y$ though $y$ is also $\G$-invariant. Let $\mathfrak{g}\subseteq \coH[0]{\Y,\cT{\Y}} $ be the Lie algebra of holomorphic vector fields arising from the infinitesimal action of $\G$. The set $\Z=\{u\in \Y\,|\,\mathfrak g (u)\subseteq { \mathcal T}_u \cF\subset   { \mathcal T}_u  {\Y}\} \ni y$ is a closed algebraic subvariety of $\Y$ saturated by $\cF$. By minimality of $\Y$, we have $\Z=\Y$, and hence ${\cF}|_\Y$ is uniruled. But this contradicts \cite[Theorem 3.6]{Loray2018}.

Suppose finally that $v|_\Y = 0$. Notice that we must have $\dim \Y < \dim \X$. By induction on $\dim \X$, we can assume that
either $\cF|_\Y$ is algebraically integrable or $\cT{\cF|_\Y}$ is numerically flat.
By \cite[Theorem 5.6]{Loray2018} if $\cF|_\Y$ is algebraically integrable and  \autoref{existence_complementary_hermitian_flat_2} if $\cT{\cF|_\Y}$ is numerically flat,
$\omega_\cF$ is a torsion line bundle, and there exists a foliation $\cG$ on $\X$ everywhere transverse to $\cF$. Therefore, replacing $\X$ by a further \'etale cover, if necessary, we may assume without loss of generality that $\omega_\cF \cong \cO{X}$. Hence, there is a $2$-form $\Omega$ on $\X$ that restricts to a symplectic form on the leaves of $\cF$. In particular, the global $1$-form $\alpha:=\hook{v}\Omega$ is nonzero.
Let $a\colon \X \to \A$ be the Albanese morphism, and set $\T:=a(\X)$. Notice that $\dim \T >0$ since $\alpha\neq 0$. Moreover, $v$ must be tangent to the fibers of $\X \to \T$ since the vanishing locus $\Z\supseteq \Y$ of $v$ is nonempty by our current assumption. Suppose that $\Z$ maps onto a proper subset of $\T$. Then a general fiber $\F$ of the map $\X \to \T$ is $\mathcal{L}$-invariant and $\mathcal{L}|_{\F} \subseteq \cT{\F}$ is regular.  \autoref{existence_complementary_hermitian_flat_2} then implies that $v$ is nowhere vanishing, a contradiction.
Therefore, $\Z$ maps onto $\T$. Now, $\alpha|_{Z_{\textup{reg}}}$ vanishes identically by its very definition.
On the other hand, $\alpha|_{Z_{\textup{reg}}}$ is the pull-back of a nonzero global $1$-form on $\A$. This immediately implies $\alpha=0$, yielding a contradiction.
This completes the proof of the theorem.
\end{proof}

We may now finish the proof of our second main result (\autoref{thm:second_main} from the introduction), which describes the structure of regular foliations of dimension two with $c_1=0$ on projective manifolds.

\begin{proof}[Proof of Theorem \ref{thm:second_main}]
    According to \autoref{theorem:classification_small_rank}, $\omega_{\cF}$ is torsion, there exists a foliation $\cG$ such that $\cT{\X} = \cT{\cF} \oplus \cT{\cG}$, and  $\cT{\cF}$ is numerically flat or $\cF$ is algebraically integrable. If $\cT{\cF}$ is numerically flat, then it is hermitian flat by \autoref{thm:num_flat_main}. If $\cF$ is algebraically integrable then the existence of a finite \'etale covering trivializing $\cF$ follows from \cite[Theorem 1.4 and Remark 3.17]{globalweinstein}.
\end{proof}

\section{Consequences for Poisson geometry}
\label{sec:poisson}
The results and conjectures in this paper have some interesting consequences for compact K\"ahler Poisson manifolds, i.e.~pairs $(\X,\pi)$ of a compact K\"ahler manifold $\X$ and a holomorphic Poisson bivector $\pi \in \coH[0]{\X,\wedge^2\cT{\X}}$.  For such manifolds, the image of the ``anchor map'' $\pi^\sharp : \forms[1]{\X}\to\cT{\X}$ gives an involutive subsheaf, whose integral submanifolds are the symplectic leaves of $(\X,\pi)$.  However, this subsheaf need not be saturated, so it may not define a foliation in the stronger sense of the present paper.  Rather, we obtain a foliation $\cF$ in the present sense by taking the saturation of $\img \pi^\sharp\subset \cT{\X}$.  This may enlarge the leaves away from the locus where the rank of $\pi$ is constant.

Let $r$ be the generic rank of $\pi$.  (The rank is always even.)  Let us assume that the locus where $\pi$ has rank less than $r$  has codimension at least two.  Then $\det \cT{\cF}$ is trivialized by $\pi^{r/2}$, and hence $c_1(\cF)=0$, so our results and conjectures on foliations with numerically trivial canonical bundle apply.

\subsection{Submanifolds and subcalibrations}

We recall the definition and basic properties of subcalibrations, due to Frejlich--M\u{a}rcu\c{t}; see \cite[\S2]{globalweinstein} for details.  A \emph{subcalibration} of a holomorphic Poisson manifold $(\X,\pi)$  is a global closed holomorphic two-form $\sigma \in \coH[0]{\X,\forms[2]{\X}}$ such that the operator $\theta=\pi^\sharp\sigma^\flat \in \End{\cT{\X}}$ is idempotent, i.e.~$\theta^2=\theta$.  Here  $\sigma^\flat : \cT{\X}\to\forms[1]{\X}$ and $\pi^\sharp : \forms[1]{\X} \to \cT{\X}$ are the natural maps defined by contraction.  The image and kernel of $\theta$ then define complementary smooth foliations $\cF$ and $\cG$, respectively, so that $\cT{\X} = \cT{\cF}\oplus\cT{\cG}$. The Poisson structure then splits as $\pi  = \pi_\cF + \pi_\cG$ where $\pi_\cF \in \coH[0]{\X,\wedge^2 \cT{\cF}}$ and $\pi_\cG \in \coH[0]{\X,\wedge^2 \cT{\cG}}$ are Poisson structures on $\cF$ and $\cG$, with $\pi_\cF$ nondegenerate.

Certain Poisson submanifolds can be used to construct subcalibrations, as follows. Let $\Y \subseteq \X$ be a closed holomorphic Poisson submanifold, i.e.~a closed complex submanifold to which $\pi$ is tangent. Assume that $\pi|_\Y$ is regular.  Such regular Poisson submanifolds always exist: for instance, if $\Y \subseteq \X$ is a Poisson subvariety that is minimal with respect to inclusions, then $\Y$ is automatically smooth and $\pi|_\Y$ is automatically regular, because the degeneracy and singular loci of any Poisson variety are Poisson subvarieties~\cite[\S2]{Polishchuk1997}.

\begin{definition}
A subcalibration $\sigma$ of $(\X,\pi)$ is \emph{compatible with the Poisson submanifold $\Y$} if $\img \theta|_\Y = \img \pi^\sharp|_\Y$, i.e.~the foliation $\cF$ of $\X$ induced by the image of $\theta$ restricts to the symplectic foliation on $\Y$.
\end{definition}

The following is a generalization of \cite[Corollary 2.5]{globalweinstein}, which treated the case in which $\Y$ is a symplectic leaf.  Note that if \autoref{conj:regular-foliation-cK} is true, then the assumption that $\Y$ is cohomologically K\"ahler can be dropped.
\begin{proposition}\label{prop:subcal}
Let $\Y\subseteq \X$ be a regular Poisson submanifold. If $\X$ and $\Y$ are connected and the symplectic foliation of $\Y$ is cohomologically K\"ahler, then $(\X,\pi)$ admits a subcalibration that is compatible with $\Y$.
\end{proposition}

\begin{proof}In the special case where $\pi|_\Y$ has corank zero, i.e.~$\Y$ is a symplectic leaf, this is \cite[Corollary 2.5]{globalweinstein}. (In that case, the foliation is trivially cohomologically K\"ahler.)  The argument here is similar, with some minor changes in the details to get around the fact that $\Y$ may have many leaves, and they need not be compact.

Suppose that the rank of $\pi|_\Y$ is $2k$, and let $\cF_\Y$ be its symplectic foliation.  Let $\kahcl \in \coH[1]{\X,\forms[1]{\X}}$ be a K\"ahler class, and consider the element $\hook{\pi^k}\kahcl^{2k} \in \coH[2k]{\X,\cO{\X}}$.  Its conjugate under Hodge symmetry is a holomorphic $2k$-form $
\mu \in \coH[0]{\X,\forms[2k]{\X}}$.  Since $\pi^k|_\Y$ is a $2k$-vector defining $\cF_\Y$, it follows from \autoref{lem:nonzero-contraction} that the element $\hook{\pi^k}\mu|_\Y \in \coH[0]{\Y,\cO{\Y}}$ is a nonzero constant.  Hence by rescaling $\mu$ if necessary, we may assume that $\mu$ restricts to the leafwise Liouville volume form on every symplectic leaf of $\Y$.  Then $\sigma_0 := \tfrac{1}{(k-1)!}\hook{\pi^{k-1}}\mu \in \coH[0]{\X,\forms[2]{\X}}$ is a global holomorphic two-form that restricts to the symplectic form on every symplectic leaf of $\Y$.

Applying \cite[Lemma 2.4]{globalweinstein} to any symplectic leaf $\bL\subseteq \Y\subseteq \X$, we deduce the existence of a subcalibration $\sigma$ of $\pi$ that restricts to the symplectic form on $\bL$.  Since $\X$ is connected, the rank of the associated endomorphism $\theta = \pi^\sharp\sigma^\flat$ is constant, and must therefore be equal to $2k$.  But then $\img\theta|_\Y\subseteq \img\pi^\sharp|_\Y$ are subbundles of $\cT{\Y}$ of the same rank, and hence they must be equal, so that the subcalibration is compatible with $\Y$, as desired.
\end{proof}

\subsection{Global Weinstein splitting}

Now suppose $(\X,\pi)$ is a compact K\"ahler Poisson manifold, and let $d$ be the minimal dimension of a symplectic leaf of $\pi$, i.e.~we have $d = \min_{x \in \X} (\rank \pi(x))$.  Assuming \autoref{conj:regular-foliation-cK} that every regular foliation is cohomologically K\"ahler, we may apply \autoref{prop:subcal} to any  Poisson submanifold on which $\pi$ has rank $d$ and deduce the existence of a subcalibration of $(\X,\pi)$ that gives a Poisson splitting of the tangent bundle  $\cT{\X} = \cT{\cF} \oplus \cT{\cG}$ where $\cF$ is a symplectic foliation of dimension $d$, and $\cG$ is a Poisson foliation.  Combining this with Beauville's conjecture, we  obtain a splitting of the universal cover of $\X$ as a product of Poisson manifolds.  In other words, \autoref{conj:regular-foliation-cK}, together with Beauville's conjecture, imply \autoref{conj:weinstein} from the introduction, whose statement we now recall:
\medskip

\noindent\textbf{\autoref{conj:weinstein}.}\emph{
\PoissonSplittingConjecture}

\medskip

If $d = 0$, or $d = \dim \X$, the conjecture is trivially satisfied: simply take $\Z$ or $\Y$ to be $\X$, and the other to be a point. If $d=\dim \X - 1$, then $\pi$ is automatically regular since the rank is even; the conjecture then follows from  the classification in \cite{DruelPoisson,CalabiYauTouzet}. The conjecture also holds when $\X$ has a compact leaf with finite fundamental group as shown by the main result of our paper~\cite{globalweinstein}.

As a direct consequence of the results presented in this paper, we provide further evidence supporting this conjecture.
\begin{proposition}\label{prop:weinstein}
\autoref{conj:weinstein} holds if $d=2$ and $\X$ is projective.
\end{proposition}
\begin{proof}
By \autoref{thm:second_main} and \autoref{lem:direct-sum-cK}, the symplectic foliation of any regular Poisson submanifold of rank two is cohomologically K\"ahler.  Applying \autoref{prop:subcal}, we obtain a subcalibration of $(\X,\pi)$ for which the foliation $\cF$ is symplectic of dimension two.  Applying \autoref{thm:second_main} again, we obtain the desired splitting.
\end{proof}

\subsection{Non-emptiness of vanishing loci}

Note that \autoref{prop:subcal} implies, in particular, that if $\X$ admits a regular cohomologically K\"ahler Poisson submanifold of rank $2k$, then there exists a holomorphic two-form $\sigma$ on $\X$ such that $\sigma^k \neq 0$.  Hence \autoref{prop:subcal} and \autoref{conj:regular-foliation-cK} (that every regular foliation is cohomologically K\"ahler) together imply the following.

\medskip
\noindent\textbf{\autoref{conj:poisson-vanishing}.}
\emph{\PoissonVanishingConjecture}

\medskip

Note that if $\X$ is rational or Fano, or more generally if $\X$ is rationally connected, then $\coh[q,0]{\X} =0$ for all $q> 0$, so the conjecture predicts that any Poisson structure on such a manifold has a zero.  This conjecture is thus similar in spirit to Bondal's conjecture~\cite{Bondal1993} that if $(\X,\pi)$ is a Fano Poisson manifold, then for every $j < \dim\X/2$, the degeneracy locus where $\pi$ has rank at most $2j$ is non-empty, and has an irreducible component of dimension at least $2j+1$.  Bondal's conjecture is known to hold for Fano manifolds of dimension $\le 4$, as proven in \cite{Gualtieri2013,Polishchuk1997}.  In addition, \autoref{conj:poisson-vanishing} was proven for $\X = \PP^5$ by the third author in \cite[Theorem 6.8.5]{PymThesis}.  Thus \autoref{conj:poisson-vanishing} holds in all those cases, all of which are subsumed by the following.

\begin{proposition}\label{prop:poisson-vanishing}
\autoref{conj:poisson-vanishing} holds if $\X$ is projective and $\dim \X \leq 6$, or more generally, if $\X$ contains a closed projective analytic Poisson subspace $\Y\subseteq \X$ of dimension $\dim \Y\leq 6$.
\end{proposition}

\begin{proof}
Without loss of generality, we may assume that the subspace $\Y$ is smooth and that $\pi|_\Y$ is regular of rank $p$ for some $p \ge d$.  By \autoref{prop:subcal}, it suffices to show that the symplectic foliation of $\Y$ is cohomologically K\"ahler.  But either the rank $p$ of $\pi|_\Y$ or the corank $q=\dim\Y -p$ is at most two, and the symplectic foliation has trivial canonical bundle; these cases are covered by \autoref{ex:extreme-cohK} when $p=0$ or $q=0$, \autoref{theorem:classification_small_rank} and \autoref{lem:direct-sum-cK} when $p=2$, \autoref{prop:codimension-one-cK} when $q=1$, and \autoref{lem:codimension-two-cK} when $q=2$.
\end{proof}

Under the stronger assumption that $\X$ is Fano, we obtain further evidence for \autoref{conj:poisson-vanishing} (and hence also Bondal's conjecture):
\begin{proposition}\label{prop:Fano-vanishing}
Let $\X$ be a Fano manifold.  If either $\dim \X \leq 7$, or $\dim \X = 8$ and $b_2(\X) = 1$, then every Poisson structure on $\X$ has at least one zero.
\end{proposition}

\begin{proof}
Every Fano manifold $\X$ has $\coh[0,q]{\X}=0$ for $q>0$.  Moreover, an application of Bott's vanishing theorem as in \cite[\S9]{Polishchuk1997} or \cite[\S7.3]{Gualtieri2013} implies that a Poisson structure $\pi$ on such a manifold can never be regular.

If $\dim\X\leq 7$, then the non-regular locus of $\pi$ gives a nonempty closed analytic Poisson subspace $\Y \subset \X$ of dimension at most six, and the result follows from \autoref{prop:poisson-vanishing}.

If $\dim \X =8$, then either there is a nonempty closed Poisson subvariety of dimension at most six (in which case \autoref{prop:poisson-vanishing} applies), or $\X$ contains a smooth Poisson hypersurface $\Y\subset \X$.   Since $b_2(\X)=1$, the divisor  $\Y$ is ample and its class is a multiple of the canonical class of $\X$.  Hence the adjunction formula gives $c_1(\Y) = s \, c_1(\X)|_\Y \in \coH[2]{\Y,\mathbb{C}}$ for some rational number $s$. If $s > 0$, then $\Y$ is Fano of dimension seven and the result follows as above. On the other hand, if $s \leq 0$ then $\omega_\Y$ is pseudo-effective, and since the symplectic foliation has trivial canonical bundle, we deduce from \autoref{P:c1=0+psef} that it is cohomologically K\"ahler, so the result follows from \autoref{prop:subcal} and the vanishing $\coh[q,0]{\X}=0$ for all $q > 0$ as above.
\end{proof}

\bibliographystyle{hyperamsplain}
\bibliography{num_flat}
\end{document}